\documentclass[12pt]{amsart}

\makeatletter
\def\@citecolor{blue}
\def\@urlcolor{blue}
\def\@linkcolor{blue}
\makeatother

\usepackage{graphicx}
\usepackage[headings]{fullpage}
\usepackage{lmodern}
\usepackage[T1]{fontenc}
\usepackage{amsmath,amsthm,amssymb,amsfonts,latexsym,mathrsfs}
\usepackage{paralist}
\usepackage{color}
\usepackage{ifthen}
\usepackage{wrapfig}
\usepackage[normalem]{ulem}

\makeatletter

\@addtoreset{equation}{section}
\def\theequation{\thesection.\@arabic \c@equation}
\def\@citecolor{blue}
\def\@urlcolor{blue}
\def\@linkcolor{blue}
\def\theenumi{\@roman\c@enumi}

\makeatother

\theoremstyle{plain}
\newtheorem{theorem}[equation]{Theorem}
\newtheorem{lemma}[equation]{Lemma}
\newtheorem{corollary}[equation]{Corollary}
\newtheorem{proposition}[equation]{Proposition}

\theoremstyle{definition}
\newtheorem{remark}[equation]{Remark}

\newtheorem{remarks}[equation]{Remarks}

\newtheorem{example}[equation]{Example}

\newtheorem{definition}[equation]{Definition}

\newtheorem{conjecture}[equation]{Conjecture}

\def\NZQ{\mathbb}               % the font for N,Z,Q,R,C
\def\NN{{\NZQ N}}

\def\frk{\mathfrak}               % font for "Fraktur"
\def\aa{{\frk a}}
\def\bb{{\frk b}}
\def\cc{{\frk c}}  
\def\mm{{\frk m}}

\def\opn#1#2{\def#1{\operatorname{#2}}} % to make operators
\opn\chara{char}
\opn\length{\ell}
\opn\projdim{proj\,dim}
\opn\depth{depth}
\opn\reg{reg}
\opn\lreg{lreg}
\opn\sat{^{sat}}
\opn\lex{^{lex}}
\opn\Ker{Ker}
\opn\Coker{Coker}
\opn\Im{Im}
\opn\Hom{Hom}
\opn\Tor{Tor}
\opn\Ext{Ext}
\opn\End{End}
\opn\Aut{Aut}
\opn\id{id}
\opn\GL{GL}
\let\:=\colon
\let\ov\overline

\opn\Gin{Gin}

\opn\Hilb{Hilb}
\opn\ini{in}
\opn\End{end}
\DeclareMathOperator{\hP}{\mathcal H}
\DeclareMathOperator{\iP}{\mathcal I}

\begin{document}
\title{The Lex-Plus-Power inequality for Local Cohomology modules}

\author{Giulio Caviglia}
\address{Giu\-lio Ca\-vi\-glia - Department of Mathematics -  Purdue University - 150 N. University Street, West Lafayette - 
  IN 47907-2067 - USA} 
\email{gcavigli@math.purdue.edu}
\author{Enrico Sbarra}
\address{Enrico Sbarra - Dipartimento di Matematica - Universit\`a degli Studi di Pisa -Largo Bruno Pontecorvo 5 - 56127 Pisa - Italy}
\email{sbarra@dm.unipi.it}
\thanks{The work of the first author was supported by a grant from the
Simons Foundation (209661 to G. C.)}

\begin{abstract}
  We prove an inequality between Hilbert functions of local cohomology modules supported in the homogeneous maximal ideal of standard graded algebras over a field, within the framework of embeddings of posets of Hilbert functions. As a main application,  we prove an analogue for local cohomology of Evans' Lex-Plus-Power Conjecture for Betti numbers. This results implies some cases of the classical Lex-Plus-Power Conjecture, namely an inequality between extremal Betti numbers.
%located at the corners of the Betti table \giulio{meglio usare extremal..? non so..Ne parliamo?}. 
In particular, for the classes of ideals for which the Eisenbud-Green-Harris Conjecture is currently known, the projective dimension and the Castelnuovo-Mumford regularity of a graded ideal do not decrease by passing to the corresponding Lex-Plus-Power ideal.
\end{abstract}

\date{\today}

\maketitle

%%%%%%%%%%%%%%%%%%%%%%%
%%%%%%%%%%%%%%%%%%%%%%%
%%%%%%%%%%%%%%%%%%%%%%%
%%%%%%%%%%%%%%%%%%%%%%%
\section*{Introduction}
The Eisenbud-Green-Harris ({\bf EGH}) and Evans' Lex-Plus-Power ({\bf LPP}) conjectures are two open problems in Algebraic Geometry and Commutative Algebra which are challenging and of great interest to  researchers in these fields, cf. \cite{EiGrHa1}, \cite{EiGrHa2},   \cite{MePeSt}, \cite{MeMu}, \cite{FrRi}, \cite{Pe}, \cite{CaMa}, \cite{Ch}, \cite{Fr},  \cite{Ri}, \cite{RiSa}. The survey \cite{PeSt}, which includes  the two conjectures above, might offer the interested reader an overview of questions which are currently considered to be significant for the classification of Hilbert functions and the study of modules of syzygies. A milestone on this subject is yielded by the work of Macaulay, cf. \cite{Ma} or the dedicated sections in \cite{BrHe}, where the sequences of numbers which are possible Hilbert functions of standard graded algebras over a field are characterized, the characterization having made possible by the introduction of a special class of monomial ideals called lexicographic ideals (lex-segment ideals for short). More precisely, all the possible Hilbert functions of  standard graded algebras  over a field  are attained by quotients of  polynomial rings by  lex-segment ideals. 
 In the 60's, some forty years after Macaulay's work,   Kruskal-Katona Theorem \cite{Kr}, \cite{Ka} provided another fundamental classification result, that of $f$-vectors of simplicial complexes, and shortly after it was generalized by the Clements-Lindstr\"om Theorem \cite{ClLi}. Both  Kruskal-Katona and  Clements-Lindstr\"om Theorems  extend Macaulay Theorem from graded quotients of $A=K[X_1,\dots,X_n]$ to graded quotients of $R=A/\aa$, where $\aa=(X_1^{d_1},\ldots,X_r^{d_r})$ with $d_1=\cdots=d_r=2$ and $2\leq d_1\leq\cdots\leq d_r$ resp., by stating that all the possible Hilbert functions  are those attained by quotients whose defining ideals are images  in $R$  of lex-segments ideals of $A$. Inspired by these results and driven by the of generalizing the famous Cayley-Bacharach Theorem, Eisenbud, Green and Harris \cite{EiGrHa1}, \cite{EiGrHa2}  conjectured among other things what is currently known as {\bf EGH}: Let ${\bf f}=f_1,\ldots,f_r$ be a regular sequence of homogeneous polynomials of degrees $d_1\leq\cdots\leq d_r$ in  $A$. Then, for every homogeneous ideal $I\subset A$ containing ${\bf f}$, there exists a lex-segment ideal $L$ of $A$ such that  $I$ and $L + (X_1^{d_1},\ldots,X_r^{d_r})$ have the same Hilbert function. Since then, the conjecture has been proven only in some special cases, cf. \cite{CaMa}, \cite{Ch} and \cite{FrRi}.\\
Furthermore, starting in the early 90's, lex-segment ideals - and other monomial ideals with strong combinatorial properties - have been studied extensively; properties of lex-segment ideals have been generalized
and studied also in other contexts, see for instance  \cite{ArHeHi1}, \cite{ArHeHi2},  generating a very rich literature on the subject. Among other results of this kind, we recall the following ones. The lex-segment ideal in the family of all homogeneous ideals with a given Hilbert function has largest Hilbert functions of local cohomology modules, which was proved in \cite{Sb}. Also, it has largest   graded Betti numbers, as it was shown  \cite{Bi}, \cite{Hu} and \cite{Pa}. In a different direction these are other extensions of Macaulay's result. Evans' Lex-Plus-Power Conjecture extends in this sense the Eisenbud-Green-Harris Conjecture to Betti numbers, by asking whether, in case {\bf EGH} holds true, a Bigatti-Hulett-Pardue type of result holds as well, i.e. {\bf LPP}: Suppose that a regular sequence $\mathbf{f}$ verifies {\bf EGH};   then the graded Betti numbers over $A$ of every homogeneous ideal $I$ containing ${\bf f}$ are smaller than or equal to those of $L + (X_1^{d_1},\ldots,X_r^{d_r})$. The conjecture is known in some few cases, the most notable one is when ${\bf f}$ is a monomial regular sequence, which is solved first in \cite{MePeSt} when $d_1=\dots=d_r=2$  and then in general in \cite{MeMu}.\\
We now take a step back, and recall that the graded Betti numbers $\beta^A_{ij}(A/I)$   are precisely the dimension as a $K$-vector space of $\Tor_i^A(A/I,K)_j$. In other words, once we have fixed $i$, we can think the sequence of $\beta^A_{ij}(A/I)$ as the Hilbert function of $\Tor_i^A(A/I,K)$, which is computed by means of a minimal graded free resolution of $A/I$. If we let $\Hilb\left(M\right)$ denote the Hilbert series of a graded module $M$, then we may restate {\bf LPP} as coefficient-wise inequalities between the Hilbert series of such $\Tor$'s: if ${\bf f}$ satisfies {\bf EGH} then for all homogeneous ideals $I$ of $A$ containing ${\bf f}$ 
$$\phantom{a}({\bf LPP})\phantom{aaaaa} \Hilb\left(\Tor_i^A(A/I,K)\right)\leq \Hilb\left(\Tor_i^A(A/(L+(X_1^{d_1},\ldots,X_r^{d_r})),K)\right),\phantom{a} \hbox{ for all } i.\phantom{a}$$
In this paper we study  analogous inequalities for Hilbert series of local cohomology modules. Let $H^i_{\mm}(\bullet)$ denote the $i^{\rm th}$ local cohomology module of a graded object with support in the graded maximal ideal. One of our main results is  Theorem \ref{mainLPP}, where we show that, if the image of ${\bf f}$ in a suitable quotient ring of $A$ satisfies {\bf EGH} (as it does in all the known cases \cite{ClLi}, \cite{CaMa}, \cite{Ch} and \cite{Ab}) then for all  homogeneous ideals $I$ of $A$ containing ${\bf f}$ 
\[
(\text{\bf {Theorem}$\;$\ref{mainLPP}}) \phantom{aaa}\Hilb\left(H^i_{\mm}(A/I)\right)\leq \Hilb\left(H^i_{\mm}(A/(L+(X_1^{d_1},\ldots,X_r^{d_r})))\right), \phantom{aa} \hbox{ for all } i.\phantom{aa}
\]
Our approach makes use of {\em embeddings of Hilbert functions}, which have been recently introduced by the first author and Kummini in \cite{CaKu} with the intent of finding a new path to the classification of Hilbert functions of quotient rings. Since ${\bf EGH}$ may be rephrased by means  of embeddings, as we explain in Section 2,  it is natural to study inequalities of the above type in this generality. In this setting we prove our main result Theorem \ref{main}, which implies Theorem \ref{mainLPP}. We let $R, S$ be standard graded $K$-algebras such that $R$ embeds into $(S,\epsilon)$, see Definition \ref{embedda}. We also assume that,  for all homogeneous ideals $I$ of $R$, 
 $\Hilb\left(H^i_{\mm_{R}}(R/I)\right)\leq \Hilb\left(H^i_{\mm_{S}}(S/\epsilon(I))\right)$ for all $i$. Then, Theorem \ref{main} states that the polynomial ring $R[Z]$ embeds into $(S[Z],\epsilon_1)$ and, for all homogeneous ideals $J$ of $S[Z]$ and for all $i$,  
$$ \Hilb\left(H^i_{\mm_{R[Z]}}(R[Z]/J)\right)\leq \Hilb\left(H^i_{\mm_{S[Z]}}(S[Z]/\epsilon_1(J))\right).$$

\noindent
Finally, in Theorem \ref{extremalLPP} we prove   {\bf LPP} for extremal Betti numbers: under the same assumption of Theorem \ref{mainLPP} for all homogeneous ideals $I$ of $A$ containing ${\bf f}$ and for all corners $(i,j)$ of $A/L+(X_1^ {d_1},\ldots,x_r^{d_r})$ we have
$$(\text{\bf {Theorem}$\;$\ref{extremalLPP}})\phantom{aaaaaaaaaaaaa} 
\beta^A_{ij}(A/I)\leq \beta^A_{ij}(A/L+(X_1^ {d_1},\ldots,x_r^{d_r})). \phantom{aaaaaaaaaaaaaaa}$$
\noindent
This paper is structured as follows. In Section 1 we introduce some general notation and we discuss the basic properties of certain ideals called $Z$-stable, together with all the related technical results needed, such as  distractions. 
In Section 2 we provide a brief summary of embeddings of Hilbert functions and we recall in  Theorem \ref{hyperplane} a General Restriction Theorem type of result proved in \cite{CaKu}.  This is aimed at setting the general framework for our main theorem and leads to the proof of Proposition \ref{hyperembed}, which is the other main tool we need.
Section 3 is devoted to our main theorem, Theorem \ref{main}. We show there that if a ring $R$ admits an embedding of Hilbert functions and its embedded ideals $\epsilon(I)$ maximize all the Hilbert functions of local cohomology modules $H^i_{\mm}(R/\bullet)$, the same is true for any polynomial ring with coefficients in $R$.
In Section 4 we explain how to derive from Theorem \ref{main} our main corollary, Theorem \ref{mainLPP}, a lex-plus-power type inequality for local cohomology which justifies the title.
Finally, in the last section we prove the validity of {\bf LPP} for extremal Betti numbers in Theorem \ref{extremalLPP} and we show in Theorem \ref{inc-region} an inclusion between the  region of the Betti table  outlined by the extremal Betti numbers of an ideal and the one of its corresponding Lex-Plus-Power ideal.

\section{ {$Z$}-stability}

\subsection{Notation}
Let $\NN$ be the set of non-negative integers, $A=K[X_1,\ldots,X_n]$ be a polynomial ring over a field $K$ and $R=\oplus_{j\in\NN}R_j=A/\aa$ be a standard graded algebra. We consider the polynomial ring $R[Z]$  with the standard grading.  When $I$ is an ideal of $R[Z]$, we will denote by $\ov{I}$ its image in  $R$ under the substitution map $Z\mapsto 0$. With $\Hilb\left(M\right)$ we denote the Hilbert series of a module $M$ which is graded with respect to total degree and with $\Hilb\left(M\right)_j$ its $j^{\rm th}$ coefficient, i.e. the $j^{\rm th}$ value of its Hilbert function. Accordingly, $\Hilb\left(M\right)_\bullet$ will denote the Hilbert function of $M$.  We say that an ideal $I$ of $R[Z]$ is \emph{$Z$-graded} if it can be written as $\bigoplus_{h\in \mathbb N} I_{\langle h\rangle}Z^h$, where each $I_{\langle h\rangle}$ is a homogeneous ideal of $R.$ In particular a $Z$-graded ideal of $R[Z]$ is homogeneous. We will denote with $\mm_R$ and $\mm_{R[Z]}$ the homogeneous maximal ideals of $R$ and $R[Z]$ respectively.

\subsection{$Z$-stability}
 The following definition was introduced  in \cite{CaKu}.
\begin{definition}\label{zstab} Let $I= \bigoplus_{h\in \mathbb N} I_{\langle h\rangle}Z^h$ be a $Z$-graded ideal of $R[Z]$. We say that $I$ is \emph{$Z$-stable} if, for all $k\geq 0$, we have 
$I_{\langle k+1\rangle} \mm_R \subseteq I_{\langle k\rangle}.$
 \end{definition}
 
\noindent
 The simplest example of a $Z$-stable ideal is the extension to $R[Z]$ of an ideal $J$ of $R$, in which case $JR[Z]=\bigoplus_{h\in \mathbb N} J Z^h$. It is easy to see that $Z$-stable ideals are fixed under the action on $R[Z]$ of those coordinates changes of $K[X_1,\dots,X_n,Z]$ which are both homogeneous and $R$-linear.
 
%When $R=K$, the ${\bf Z}$-stable ideals of $R'$ are precisely those called \emph{strongly stable}, cf. 15.9.3 in \cite{Ei}. These ideals are invariant under the action of the {\em Borel subgroup} consisting of  upper triangular coordinates changes and,  when $\chara(K)=0$, they are the only. Such ideals play an important role in the theory of generic initial ideals, see 15.9.1 and 15.9.2 in \cite{Ei}. 

\begin{remark}\label{ZZm} 
(a) Let $I$ be a $Z$-stable ideal of $R[Z]$ and let us write the degree $d$ component $I_d$ of $I$ as a direct sum of vector spaces $V_d \oplus V_{d-1}Z \oplus \cdots \oplus V_{0}Z^d$. It follows directly from Definition \ref{zstab} that, for all $j=0,\dots,d$, the $j^{\rm th}$ component of the $R$-ideal generated by $\bigoplus_{j=0}^d V_j$  is the vector space $V_j$.\\
(b) It is immediately seen that the ideal $I:Z= \bigoplus_{h\in\NN} I_{\langle h+1\rangle}Z^h$ is $Z$-stable as well. Furthermore, we observe that $I:Z= I: \mm_{R[Z]}$; one inclusion is clear and the other follows from
$$\mm_{R[Z]}(I:Z)=\bigoplus_{h\in\NN}I_{\langle h+1\rangle}(\mm_{R}+Z)Z^h
\subseteq \bigoplus_{h\in\NN}I_{\langle h\rangle}Z^h\ + \bigoplus_{h\in\NN}I_{\langle h+1\rangle}Z^{h+1}
\subseteq I.$$ 
In particular, $ I\sat:= \bigcup_i (I: \mm_{R[Z]}^i)= \bigcup_i (I: Z^i)  =:I:Z^{\infty}$ is a $Z$-stable ideal.
\end{remark}

\noindent
The next result about $Z$-stable ideals  will be used later in the paper.

\begin{lemma}\label{Zmodulo}
Let $I$ and $J$ be $Z$-stable ideals of $R[Z]$ with $\Hilb\left(I\right)_i=\Hilb\left(J\right)_i$ for all $i\gg 0$.
Then, $\Hilb\left(\ov{I}\right)_j=\Hilb\left(\ov{J}\right)_j$ for all $j\gg 0$.
\end{lemma}
\begin{proof}
Since $i\gg 0$, we may assume that there is no generator of $I$ or $J$ in degree $i$ and above.
As in Remark \ref{ZZm} (a), we denote the vector space of all homogeneous polynomials in $I$ of degree $i$ by $I_i$, and we write it as direct sum of vector spaces $V_i\oplus V_{i-1}Z\oplus\cdots\oplus V_0Z^i.$ From the definition of $Z$-stability it follows that $I_i(\mm_{R[Z]})_1\subseteq V_i(\mm_R)_1\oplus I_iZ$. Hence, we get a decomposition of the vector space $I_{i+1}$ as direct sum $\ov{I}_i(\mm_R)_1\oplus I_{i}Z=\ov{I}_{i+1}\oplus I_{i}Z.$ Similarly, we can write $J_{i+1}$ as $\ov{J}_{i+1}\oplus J_{i}Z$ and now the conclusion follows easily from the hypothesis.  
\end{proof}

\subsection{Distractions} Let $I$ be a $Z$-graded ideal of $R[Z]$ and let $l=\ov{l}+Z$ be an element of $R[Z]$ with $\bar{l}\in R_1$ (and possibly zero). Given a positive degree $d$ we define the \emph{$(d,l)$-distraction} of $I$, and we denote it by $D_{(d,l)}(I)$,
$$D_{(d,l)}(I):=\bigoplus_{0\leq h<d } I_{\langle h\rangle}Z^h \oplus l\left(\bigoplus_{d\leq h} I_{\langle h\rangle}Z^{h-1}\right).$$
\noindent
It is not hard to see that $D_{(d,l)}(I)$ is an ideal of $R[Z]$ with the same Hilbert function as $I$, see \cite{CaKu} Lemma 3.16 and also \cite{BiCoRo} for more information about general distractions of the polynomial ring $A$. To our purposes, it is important to notice that $D_{(d,l)}(I)$ can be realized in two steps as a polarization of $I$  followed by a specialization. In order to do so, we first define $J$ to be the ideal of $R[Z,T]$ generated by $\bigoplus_{0\leq h<d } I_{\langle h\rangle}Z^h \oplus T\left(\bigoplus_{d\leq h} I_{\langle h\rangle}Z^{h-1}\right).$ Notice that the Hilbert function of $J$ is the same as the one of $IR[Z,T]$. This implies that  $T-Z$ and $T-l$ are $R[Z,T]/J$-regular, and thus there exist isomorphisms 
\begin{equation}\label{dis}
R[Z,T]/(J+ (T-Z)) \simeq R[Z]/I,\;\;\;\;\;\; R[Z,T]/(J+ (T-l)) \simeq R[Z]/D_{(d,l)}(I). 
\end{equation}
  
We now define a partial order $\prec$  on all the $Z$-graded ideals of $R[Z]$ by letting 
\begin{equation}\label{order}
J\preceq L \hbox{\;\;\; iff\;\;\; } 
\Hilb\left(\bigoplus_{k\leq h}J_{\langle k \rangle} Z^k\right)\;\leq\; \Hilb\left(\bigoplus_{k\leq h} L_{\langle k \rangle }Z^k\right) \hbox {\;\;\;\;\;for all } h,
\end{equation}
where $\bigoplus_{k\leq h}J_{\langle k \rangle} Z^k$ and $\bigoplus_{k\leq h} L_{\langle k \rangle }Z^k$ are considered as graded $R$-modules with $\deg Z^k=k$. 
 We write $J\prec L$ when $J\preceq L$ and at least one of the above inequalities  is strict.
Let now $I$ be a  $Z$-graded ideal of $R[Z]$ and consider the partially ordered set  $\mathcal{I}$ of all $Z$-graded ideals of $R[Z]=A/\aa[Z]$ with the same Hilbert function as $I$. We claim that $\mathcal{I}$ has finite dimension as a poset, i.e. the supremum of all lengths of  chains in $\mathcal{I}$ is finite. To this end, we fix a monomial order $\tau$ on $A$ and we compute the initial ideal with respect to $\tau$ of the pre-image in $A[Z]$ of every ideal in $\mathcal I$.  In this way, we have constructed a set $\mathcal{J}$ of monomial ideals in $A[Z]$  all with the same Hilbert function, say $H$. By Macaulay Theorem, this set is finite, for the degrees of the minimal generators of an ideal in $\mathcal J$ are bounded above by the degrees of the minimal generators of the unique lex-segment ideal with Hilbert function $H$. Since every chain of $\mathcal I$ lifts to a chain in $\mathcal{J}$, our claim is now clear.

\begin{proposition}\label{reduction}
Let $I$ be a $Z$-graded ideal of $R[Z]$ and let $\omega=(1,\dots,1,0)$ be a weight vector. If $I$ is not $Z$-stable, then there exist a positive integer $d$ and a linear form $l=\ov{l}+Z$  with $\bar{l}\in R_1$ such that $I\prec \ini_\omega(D_{(d,l)}(I))$. 
\end{proposition}
\begin{proof} We write $I$ as $\bigoplus_{h\in \mathbb N} I_{\langle h\rangle}Z^h$ and we let $d>0$ be the least positive integer such that $I_{\langle d\rangle} \mm_R \not \subseteq I_{\langle d-1\rangle}$. Thus,  there exists an  
indeterminate $X_j$ of $A$ such that $I_{\langle d \rangle}X_j \not \subseteq I_{\langle d-1 \rangle}$; therefore  we can define $l$ to be the image in $R[Z]$ of $X_j+Z$ and $D_{(d,l)}$ to be the corresponding $(d,l)$-distraction. Since $D_{(d,l)}(\bigoplus_{j \leq i}R Z^j)\subseteq \bigoplus_{j \leq i}R Z^j $ for all $i$, we have  $I\preceq \ini_\omega(D_{(d,l)}(I))=:J$. Furthermore, $I_{\langle d-1\rangle} \subseteq J_{\langle d-1\rangle}$ and $I_{\langle d\rangle}X_j \subseteq J_{\langle d-1\rangle}.$
Since $ I_{\langle d \rangle} X_j\not \subseteq I_{\langle d-1 \rangle}$  we deduce that 
$I_{\langle d-1\rangle} \subsetneq J_{\langle d-1\rangle}$, hence  $I\prec J$ as desired.
\end{proof}

From now on, $H^i_\mm(\bullet)$ will denote the $i^{\rm th}$ local cohomology with support in the homogeneous maximal ideal $\mm$ of a standard graded $K$-algebra.

\begin{proposition}\label{rred}  Let $I$ be a homogeneous ideal of $R[Z]$. Then, there exists a $Z$-stable ideal $J$ of $R[Z]$ with the same Hilbert function as $I$ such that   
$$\Hilb\left(H^i_{\mm_{R[Z]}}(R[Z]/I)\right)\leq \Hilb\left(H^i_{\mm_{R[Z]}}(R[Z]/J)\right), \hbox{\;\;\; for all } i.$$ 
\end{proposition}
\begin{proof} 
By \cite{Sb} Theorem 2.4,  
$\Hilb\left(H^i_{\mm_{R[Z]}}(R[Z]/I)\right)\leq \Hilb\left(H^i_{\mm_{R[Z]}}(R[Z]/\ini_\omega(I))\right)$ where $\omega=(1,\dots,1,0)$, and thus without loss of generality we may assume that $I$ is $Z$-graded. Let now $\mathcal I$ be the set of all $Z$-graded ideals of $R[Z]$ with same Hibert function as $I$ and $J\in \mathcal{I}$ be maximal - with respect to the partial order $\prec$ defined in \eqref{order} - among the ideals of $\mathcal{I}$ which  satisfy $\Hilb\left(H^i_{\mm_{R[Z]}}(R[Z]/I)\right)\leq \Hilb\left(H^i_{\mm_{R[Z]}}(R[Z]/J)\right)$ for all $i$. We claim that $J$ is $Z$-stable. If it were not, by Proposition \ref{reduction}  there would  exist a positive integer $d$ and a linear form $l$ such that $J\prec \ini_\omega(D_{(d,l)}(J))$. By \eqref{dis} together with \cite{Sb} Section 5, we have  
$\Hilb\left(H^i_{\mm_{R[Z]}}(R[Z]/J)\right)=\Hilb\left(H^i_{\mm_{R[Z]}}(R[Z]/D_{(d,l)}(J))\right)$. By \cite{Sb} Theorem 2.4, $\ini_\omega(D_{(d,l)}(J))\in\mathcal{I}$, contradicting the maximality of $J$.
\end{proof}

 \section{Embeddings and the General Hyperplane Restriction Theorem}
Let $B=K[X_1,\dots,X_n]$ be a standard graded polynomial ring over a field $K$,  $\bb$ a homogeneous ideal of $B$ and $S=B/\bb.$
Denote by $\iP_{S}$  the poset $\{J \: J \ \text{is a homogeneous $S$-ideal}\}$ ordered by inclusion, and with $\hP_{S}$  the poset
$\{\Hilb\left(J\right)_\bullet \: J \in \iP_{S}\}$ of all Hilbert functions of the ideals in $\iP_{S}$ with  the usual point-wise partial order. Following \cite{CaKu} we say that $S$  \emph{admits an embedding} if there exists an order preserving injection $\epsilon\:\hP_{S} \longrightarrow \iP_{S}$ such that the image of any given Hilbert function is an ideal with that Hilbert function. We call any such $\epsilon$ an \emph{embedding of $S$}. A ring $S$ with a specified embedding $\epsilon $ is denoted by $(S,\epsilon)$ and, for simplicity's sake, we also let $\epsilon(I):=\epsilon(\Hilb\left(I\right)_\bullet)$, for every $I\in \iP_{S}$. The notion of  embedding  captures the key property of rings for which an analogous of Macaulay Theorem  holds. 

\begin{example}[Three standard examples of embedding]\label{thhree} In the following we present some results that
can be re-interpreted with the above terminology.\\
(a) Let $S=B$ and define $\epsilon\: \mathcal{H}_S\longrightarrow \mathcal{I}_S$ as $\epsilon(\Hilb\left(I\right)_\bullet):=L$, 
where $L$ is the unique lex-segment ideal with Hilbert function $\Hilb\left(I\right)_\bullet$. The fact that this map is well-defined is just a restatement of Macaulay Theorem.\\ 
(b) Let $S=B/\bb$, where $\bb=(X_1^{d_1},\ldots,X_r^{d_r})$ and $d_1\leq\cdots \leq d_r$. Let $I$ be a homogeneous ideal of $S$, by Clements-Lindstr\"om Theorem \cite{ClLi} there exists a lex-segment ideal $L\subseteq B$ such that $\Hilb\left(I\right)=\Hilb\left(LS\right)$. Since $LS$ is uniquely determined by $\Hilb\left(I\right)$,  we may define  $\epsilon(\Hilb\left(I\right)_\bullet):=LS$. The ideal $L+\bb$, which is uniquely determined by the Hilbert function of $I$, is often referred to as  \emph{the Lex-Plus-Power ideal associated with $I$ (with respect to $d_1,\dots,d_r$)}.\\
(c) Let $m$ be a positive integer and $S=B^{(m)}=\bigoplus_{d\geq 0}B_{md}$ the $m^{\rm th}$-Veronese subring of $B$. Let $I$ be a homogeneous ideal of $S$. Then, by \cite{GaPeMu}, there exists a lex-segment ideal $L$ of $B$ such that $\Hilb\left(I\right)=\Hilb\left(\bigoplus_{d\geq 0}L_{md}\right)$ and we define $\epsilon(\Hilb\left(I\right)_\bullet):=\bigoplus_{d\geq 0}L_{md}$.\\
We observe that in all of the above examples the image set of $\epsilon$ consists of the classes in $S$ of lex-segment ideals; rings with this property  are called {\em Macaulay-Lex}, cf. for instance \cite{MePe}.

These examples  can be derived by general properties of embeddings proved in \cite{CaKu}. For instance, if we let $(S,\epsilon)$ be a ring with an embedding then:\\
(i) the polynomial ring $S[Z]$ admits an embedding and if $\epsilon$ is defined by means of a lex-segment as in (a) and (b) then so is the extended embedding on $S[Z]$;\\ 
(ii) when $\mathcal{H}_{S[Z]/(Z^d)}= \left \{ \Hilb\left(JS[Z]\right)_\bullet \: J \hbox{\; is }Z\hbox{-stable}\right\}$, the ring $S[Z]/(Z^d)$ admits an embedding as well. By an iterated use of this fact, starting with $S=K$, one can recover (b);\\
(iii) any Veronese subring $S^{(m)}$ of $S$ admits an embedding inherited from $(S,\epsilon)$;\\ 
(iv) $S/\epsilon(I)$ admits an embedding induced by $\epsilon$;\\
(v) when $\bb$ is monomial and $\cc\subseteq T$ is a polarization of $\bb$, then 
$T/ \cc$ admits an embedding.\hfill{\qed}
\end{example}

%- For the sake of notational simplicity, the restriction of $\epsilon$ to the set ${\{ \Hilb(J) \:  J\in \iP_S \}}$ will also be denoted by $\epsilon.$ \giulio{questa riga e' sbagliata. Dovrebbe venire dopo le cose sotto. Non e' una restrizione!}\\Let now $B=K[X_1,\dots,X_m]$ be a second polynomial ring, $R$ and $S$ be standard graded algebras over a field $K$, we say $R=A/\aa$, $S=B/\bb$. 

%\begin{remark}
%Observe that for every homogeneous ideal $I$ of $R$ there exists an ideal of $S$, we say $J$,  such that %$R/I$ and $S/J$ have the same Hilbert function.  If $R$ embeds into $(S,\epsilon)$ we also let  %$\epsilon(I):=\epsilon(J)$, for any $I\in\iP_R$ and $J\in\iP_S$ such that $\Hilb(R/I)=\Hilb(S/J)$, after %observing that this definition does not depend on J and is therefore well-posed.
%\noindent
% In this way we can see $\epsilon$ also as a (not necessarily injective) order-preversing function between %the posets $\iP_R$ and $\iP_S$, which also preserves Hilbert functions of the corresponding quotient rings.
%\end{remark}

The following theorem generalizes to rings with embedding \cite{HePo} Theorem 3.7 (see also \cite{Ga} Theorem 2.4) valid for polynomial rings.

\begin{theorem}[General Restriction Theorem]\label{hyperplane}
Let $(S,\epsilon)$ be a standard graded $K$-algebra with an embedding and $S[Z]$ a polynomial ring in one variable with  coefficients in $S$.
There exists an embedding $\epsilon_1\:\mathcal{H}_{S[Z]}\longrightarrow\mathcal{I}_{S[Z]}$ such that $\epsilon_1(I)=\bigoplus_{h}J_{\langle h\rangle}Z^h$  is a $Z$-stable ideal with $\epsilon(J_{\langle h\rangle})=J_{\langle h\rangle}$. Moreover, if $I$ is $Z$-stable, then
\[
\Hilb\left(I+(Z^j)\right)\geq\Hilb\left(\epsilon_1(I)+(Z^j)\right), \hbox{\;\;\; for all } j.
\]
\end{theorem}
\begin{proof}
See that of \cite{CaKu} Theorem 3.9 (see also \cite{CaKu2} Theorem 2.1).
\end{proof}

\begin{definition}\label{embedda}  Let $(S,\epsilon)$ be a ring with an embedding. We say that a $K$-algebra $R$ {\it embeds into} $(S,\epsilon)$ and we write $(R,S,\epsilon)$ if $\mathcal{H}_R\subseteq\mathcal{H}_S$. Moreover, we say that an ideal
$I$ of $S$ is embedded if it is in the image of $\epsilon$.  For simplicity's sake, we let again $\epsilon(I):=\epsilon(\Hilb\left(I\right)_\bullet)$ for all homogeneous ideals $I$ of $R$.
\end{definition}
\noindent
Clearly, $(S,\epsilon)$ embeds into itself and if $R$ embeds into $(S,\epsilon)$ then $\Hilb\left(R\right)=\Hilb\left(S\right)$, for there is an ideal $I\in\mathcal{I}_S$ with same Hilbert function as $R$, therefore $\Hilb\left(I\right)_0=1$ implies $I=S$.

The above definition is motivated by the conjecture of Eisenbud, Green and Harris, which has been discussed in the introduction. When a regular sequence ${\bf f}$ of $A=K[X_1,\dots,X_n]$ satisfies {\bf EGH}, the Hilbert functions of homogeneous ideals in $R=A/({\bf f})$ are also Hilbert functions of homogeneous ideal in $S=B/(X_1^{d_1},\ldots,X_r^{d_r}).$ Equivalently   $\mathcal{H}_R\subseteq\mathcal{H}_S.$ The fact that Clements-Lindstr\"om Theorem holds for $S$ tells us that $R$ embeds into $S.$ One of the reasons why {\bf EGH} is important is that  $\mathcal{H}_R\subseteq\mathcal{H}_S$ together with the fact that Clements-Lindstr\"om Theorem gives an embedding, allow to transfer certain results, e.g. an  uniform upper bound on the number of generators  as in Remark \ref{gen}), from the ring  $S$ to the ring $R$. 

%The ring $S$ has several interesting combinatorial properties: the Clements-Lindstrom theorem and the result in \cite{MeMu} are two of them...............\\

Let now $(S,\epsilon)$ be a ring with an embedding, and let $I\in\mathcal{I}_S.$ It is easily seen from \cite{CaKu} Proposition 2.4 and Definition 2.3 (i) that the ideal $\epsilon(I)\mm_S$ is embedded i.e. $\epsilon(\epsilon(I)\mm_S)=\epsilon(I)\mm_S$.

\begin{lemma}\label{isthisthelast} 
Let $R$ embed into $(S,\epsilon)$ and let $I\in\mathcal{I}_R$. Then, $\mm_S\epsilon(I)\subseteq \epsilon(\mm_RI)$. 
\end{lemma}
\begin{proof}
Since $\mm_S\epsilon(I)$ is embedded and embeddings preserve poset structures we only need to show that $\Hilb\left(\mm_S\epsilon(I)\right)_\bullet\leq \Hilb\left(\epsilon(\mm_RI)\right)_\bullet$ or, equivalently, that 
$\dim_K (S_1\epsilon(I))_{d+1}\leq\dim_K R_1I_d$ for all $d\geq 0$. Since $\epsilon(I)$ contains $\epsilon((I_d))$ and  they agree in degree $d$, we have $(S_1\epsilon(I))_{d+1}=(S_1\epsilon((I_d)))_{d+1}$ and its dimension is smaller than or equal to that of $\epsilon((I_d))_{d+1}$. Now it is enough to observe that the latter has the same dimension as $R_1I_d$. 
\end{proof}

\begin{remark} \label{gen} There is point-wise inequality between the number and the degrees of minimal generators of a homogeneous ideal $I\subseteq R$ and the ones of $\epsilon(I),$ namely $\beta_{1j}^{R}(I)\leq  \beta_{1j}^{S}(\epsilon(I)).$ This inequality is equivalent to $\Hilb \left(I/\mm_R I\right) \leq 
\Hilb\left(\epsilon(I)/ \mm_S\epsilon(I)\right),$ which follows immediately from  Lemma \ref{isthisthelast}.
\end{remark}

\begin{proposition} \label{hyperembed}
Let $R$ embed into $(S,\epsilon)$ and let $\epsilon_1$ as in Theorem \ref{hyperplane}. Then $R[Z]$ embeds into $(S[Z],\epsilon_1)$. Moreover, if $I$ is a $Z$-stable ideal of $R[Z]$ then
\[
\Hilb\left(I+(Z^j)\right)\geq\Hilb\left(\epsilon_1(I)+(Z^j)\right), \hbox{\;\;\; for all } j.
\]
\end{proposition}
\begin{proof} 
  Let $I$ be a homogeneous ideal of $R[Z]$.  By Proposition \ref{rred}  there exists a $Z$-stable ideal of $R[Z]$ with the same Hilbert function as $I$ so that we may assume that $I$ is $Z$-stable. Now we write $I$ as $\bigoplus_{h\in\NN} I_{\langle h \rangle} Z^h$ and we let $J$ be the $S$-module  $\bigoplus_{h\in\NN} \epsilon(I_{\langle h \rangle}) Z^h$. It is easy to see that $J$ is an ideal of $S$, for $\epsilon(I_{\langle 0 \rangle})\subseteq \epsilon(I_{\langle 1 \rangle})\subseteq\cdots$, and thus  $\mathcal{H}_{R[Z]}\subseteq\mathcal{H}_{S[Z]}$.  By Lemma \ref{isthisthelast}, $\mm_S\epsilon(I_{\langle h+1 \rangle})\subseteq\epsilon(I_{\langle h \rangle})$  which implies that $J$ is $Z$-stable. We now have $\Hilb\left(I+(Z^j)\right)=\Hilb\left(J+(Z^j)\right)$ for all $j$ and we can conclude the proof by applying Theorem \ref{hyperplane} since $\epsilon_1(J)=\epsilon_1(I)$. 
\end{proof}

%\begin{notrem}\label{extemb}
%Given rings $R$ and $S$ and an embedding of $R$ into $(S,\epsilon)$, one can extend it by  an iterated use of Theorem \ref{hyperplane} and Corollary \ref{hyperembed} to an embedding of $R'$ into $(S',\epsilon')$, by letting at each step $S'_{[j]}:=S[Z_1,\ldots,Z_j]$, $\epsilon'_j:=(\epsilon'_{j-1})'$ and $R'_{[j]}:=R[Z_1,\ldots,Z_j]$, so that $R'_{[m]}$ embeds into $(S',\epsilon')=(S'_{[m]},\epsilon'_m)$.
%It is important to notice that by Remark \ref{ZZm} (a) all the ideals in the image of $\epsilon'$ are $\bf{Z}$-stable.

%Let  $I$ be a homogeneous ideal of $S'=S[Z_1,\ldots,Z_m]$. We denote by $I_{[h]}$ the ideal of $S'_{[h]}$ generated by all elements in $I$ evaluated at $Z_{h+1}=\cdots=Z_{m}=0$. We let $I\sat$ denote the saturation of $I$ with respect to the homogeneous maximal ideal of  $S'$, whereas $I_{[h]}^{\sat_h}$ will denote the saturation of $I_{[h]}$ in the appropriate polynomial ring, i.e. we let $\mm_{[h]}:=\mm_{S_{[h]}}$ and 
%$I_{[h]}^{\sat_{h}}:=I_{[h]}:\mm_{[h]}^\infty$. Accordingly, $I\sat=I_{[m]}^{\sat_m}$. We will use an analogous notation for ideals of $R'.$

%Finally we recall that, if $I$ is  $\mathbf{Z}$-stable, its saturation is equal to the colon ideal of $I$ with a sufficiently high power of $(Z_m)$, as we have proven in \eqref{satOK}.     
%\end{notrem}

%\noindent
%The next two results are going to be of crucial importance for the proof of the main theorem.

We conclude this section with a technical result we need later on.

\begin{lemma}\label{pregiatolemmaZ2palle}
Let $R$ embed into $(S,\epsilon)$ and let $\epsilon_1$ as in Theorem \ref{hyperplane}. If $I$ is a $Z$-stable ideal of $R[Z]$, then $\epsilon(\ov{I})\sat=\ov{\epsilon_1(I)}\sat$.
\end{lemma}
\begin{proof}
First, we observe that $\ov{\epsilon_1(I)}=\epsilon_1(I)_{\langle 0\rangle}$ is an embedded ideal of $R$ by Theorem \ref{hyperplane}.  By Proposition \ref{hyperembed}, $\Hilb\left(\ov{I}\right)_\bullet\geq\Hilb\left(\ov{\epsilon_1(I)}\right)_\bullet$, hence $\epsilon(\ov{I})\supseteq \epsilon(\ov{\epsilon_1(I)})=\ov{\epsilon_1(I)}$. Moreover, $\Hilb\left(\epsilon(\ov{I})\right)_j=\Hilb\left(\ov{I}\right)_j=\Hilb\left(\ov{\epsilon_1(I)}\right)_j$ for $j\gg 0$ by Lemma \ref{Zmodulo}, from which we deduce that $\epsilon(\ov{I})_j=\ov{\epsilon_1(I)}_j$ for $j\gg 0.$ This is enough to complete the proof, since saturation of a homogeneous ideal can be computed by any of its  sufficiently high truncations.
\end{proof}

\section{The main theorem}
In this section we illustrate our main result, which is stated in the next theorem.
%Let as before $(S,\epsilon)$ be a ring with an embedding,  $\mathcal{I}_S$ be the set of graded ideals of $S$ and $\mathcal{H}_S$ the set of Hilbert functions of $\mathcal{I}_S$.
We say that $(R,S,\epsilon)$ is {\em (local) cohomology extremal} if, for every homogeneous ideal $I$ of $R$ and all $i$, one has
$\Hilb\left(H^i_{\mm_R}(R/I)\right)\leq \Hilb\left(H^i_{\mm_S}(S/\epsilon(I))\right).$

We recall that both $R$ and $S$ are graded quotients of $A=B=K[X_1,\dots,X_n]$, and that the projective dimension $\projdim_A(M)$ and the Castelnuovo-Mumford regularity $\reg_A(M)$  of a finitely generated graded $A$-module $M$ can be expressed in terms of local cohomology modules as $\max\{n-i : H^i_{\mm_A}(M) \not =0 \}$ and $\max\{d+i : H^i_{\mm_A}(M)_d \not =0 \}$ respectively.
Thus when $(R,S,\epsilon)$ is cohomology extremal, for every homogeneous ideal $I$ of $R$ one has 
$$\projdim(R/I)\leq \projdim(S/\epsilon (I)) \text{  $\quad$ and  $\quad$ }  \reg(R/I)\leq \reg(S/\epsilon (I)),$$ 
and analogous inequalities hold for the embeddings $\epsilon_1$ and $\epsilon_m$ of Theorem \ref{main} and Corollary \ref{conm} below. 
%\end{remark}

\begin{theorem}\label{main} Let $(R,S,\epsilon)$ be cohomology extremal. Then, $(R[Z],S[Z],\epsilon_1)$ is cohomology extremal. 
\end{theorem}
\noindent

By recursion, one immediately obtains the natural generalization to the case of $m$ variables.

\begin{corollary}\label{conm} Let $m$ be a positive integer and  $(R,S,\epsilon)$ be cohomology extremal. Then, $(R[Z_1,\ldots,Z_m],S[Z_1,\ldots,Z_m],\epsilon_m )$ is cohomology extremal. 
\end{corollary}

Let $A=K[X_1,\ldots,X_m]$ be a polynomial ring over a field $K$, $I\subseteq A$ a homogeneous ideal and $L$ the unique lex-segment ideal of $A$ with the same Hilbert function as $I$. It was proven in \cite{Sb} Theorem 5.4 that 
\begin{equation}\label{ENR}
\Hilb\left(H^i_{\mm_A}(A /I)\right)\leq\Hilb\left(H^i_{\mm_A}(A /L)\right), \hbox{\;\;\;for all } i.
\end{equation}
This result can be now recovered from the above corollary, since any field $K$ has a trivial embedding $\epsilon_0$, so that $(K,K,\epsilon_0)$ and $(A,A,\epsilon_m)$ are cohomology extremal. By Example \ref{thhree} part (i), we know that  $\epsilon_m(I)$ is the lex-segment ideal $L$. 

Similarly, Corollary \ref{conm} implies the following result, which is the analogous inequality, for local cohomology, of that for Betti numbers proved by Mermin and Murai in \cite{MeMu}.

\begin{theorem} 
Let $A=K[X_1,\dots,X_n]$, $\aa=( X_1^{d_i},\dots,X_r^{d_r})$, with $d_1\leq \cdots \leq d_r.$ Let $I \subseteq A$ be a homogeneous ideal containing $\aa,$ and let $\rm{L}+ \aa$ be the Lex-Plus-Power ideal associated to $I$ with respect to $d_1,\dots,d_r.$ Then,
\begin{equation}\label{LLPC}
\Hilb\left(H^i_{\mm_A}(A /I)\right)\leq\Hilb\left(H^i_{\mm_A}(A /L+\aa)\right), \hbox{\;\;\;for all } i.
\end{equation}
\end{theorem} 
\begin{proof} Let $\bar R=K[X_1,\dots,X_r]/(X_1^{d_1},\dots,X_r^{d_r})$; by Example \ref{thhree} (b) 
$\bar R$ has an embedding $\epsilon$ induced by the Clements and Lindstr\"om Theorem. 
Since $\bar R$ is Artinian we see immediately that $(\bar R,\bar R,\epsilon)$ is cohomology extremal being
$H^0_{\mm_{\bar R}}(\bar R/J)=\bar R/J$ and $H^0_{\mm_{\bar R}}(\bar R/\epsilon(J))=\bar R/\epsilon(J)$ for all homogeneous ideal $J$ of $\bar R.$
Now, $A/\aa$ is isomorphic to $\bar R[X_{r+1},\dots,X_{n}]$, and by Corollary \ref{conm} we know that $\epsilon_{n-r}$ is cohomology extremal. Furthermore, $\epsilon_{n-r}(I (A/\aa))=L(A/\aa)$ (see \cite{CaKu2} Remark 2.5). By Base Independence of local cohomology, we thus have
\[
\begin{array} {lcl} 
\Hilb\left(H^i_{\mm_A}(A /I)\right) 
                & = & \Hilb\left(H^i_{\mm_{A/ \aa}}(A /I)\right) \\
                & \leq &  \Hilb\left(H^i_{\mm_{A/ \aa}}((A/ \aa)/ \epsilon_{n-r}(I(A/ \aa)))\right)\\
                & = & \Hilb\left(H^i_{\mm_A}(A /L+\aa)\right), \hbox{\;\;\;for all } i.
\end {array}
\]
\end{proof}

For the proof of Theorem \ref{main} we need some preparatory facts. First, we observe that for any  homogeneous ideal $I$ of $R[Z]$,
\begin{equation}\label{filanoesaturano}
\ov{I\sat}\sat=\ov{I}\sat,
\end{equation}
since $\ov{I\sat}$ and $\ov{I}$ coincide in high degrees because $I\sat$ and $I$ do.\\
It is not difficult to see that, if $I$ is a homogeneous ideal of $R$, then for all $i>0$ 
\begin{equation*}\label{mia}
\Hilb\left(H^i_{\mm_{R[Z]}}(R[Z]/IR[Z])\right)_h=\sum_{k\geq h}\Hilb\left(H^{i-1}_{\mm_{R}}(R/I)\right)_{k+1},
\end{equation*}
\noindent
cf. for instance  \cite{Sb2} Lemma 2.2  for a proof. As an application,  when $I$ is a $Z$-stable ideal of $R[Z]$ and $i>0$ one has 
\begin{equation}\label{sommaecala}
%\begin{split}
\Hilb\left(H^i_{\mm_{R[Z]}}(R[Z]/I)\right)_h
=\Hilb\left(H^i_{\mm_{R[Z]}}(R[Z]/I\sat)\right)_h\\
=\sum_{k\geq h}\Hilb\left(H^{i-1}_{\mm_{R}}(R/\ov{I\sat})\right)_{k+1},
%\end{split}
\end{equation}
\noindent
which is clearly equivalent to 
\begin{equation}\label{compact}
\Hilb\left(H^i_{\mm_{R[Z]}}(R[Z]/I)\right)= \sum_{j<0} t^j\cdot\Hilb\left( H^{i-1}_{\mm_R}(R/\ov{I\sat})\right), \hbox{\;\;\;for } i>0.
\end{equation}
We shall also need the observation, yielded by \eqref{compact} together with  \eqref{filanoesaturano}, that for a $Z$-stable ideal $I$ of $R[Z]$ 
\begin{equation}\label{perIm1}
\Hilb\left(H^i_{\mm_{R[Z]}}(R[Z]/I)\right) = \sum_{j<0} t^j\cdot\Hilb\left(H^{i-1}_{\mm_{R}}(R/\ov{I}\sat)\right), \hbox{\;\;\;for } i>1.
\end{equation}

\begin{lemma}\label{perH2}
Let $I$ be a $Z$-stable ideal of $S[Z]$ and $d\gg 0$ a fixed integer. Then, for all $j=0,\ldots d$, 
$$\sum_{k=0}^j\Hilb\left(\ov{I\sat}\right)_{d-k}\leq
\sum_{k=0}^j\Hilb\left(\ov{\epsilon_1(I)\sat}\right)_{d-k}.$$ 
\end{lemma}
\begin{proof}
Let $I_d=V_d\oplus V_{d-1} Z_m \oplus   \cdots\oplus V_0Z_m^d$ be a decomposition of $I_d$ as a direct sum of vector spaces.
Since $d\gg 0$, we have that $I\sat=(I_d)\sat$ which by $Z$-stability is  $(I_d):Z_m^\infty$; therefore, $I\sat$ is generated by the elements in $V_0\oplus\cdots\oplus V_d$, which also generate in $S$ the ideal $\ov{I\sat}$. Moreover, cf. Remark \ref{ZZm}(a),  $(\ov{I\sat})_{d-j}$ is exactly the vector space  $V_j$, for $0\leq j\leq d$. The same argument can be repeated for $\epsilon_1(I)$, since it is also $Z$-stable; therefore, the two terms which appear in the inequality that has to be proven  are the values at $d$ of the Hilbert function of $I+(Z_m)^j/(Z_m)^j$ and $\epsilon_1(I) + (Z_m)^j/(Z_m)^j$  respectively, thus the conclusion follows now immediately from Theorem \ref{hyperplane}. 
\end{proof}
 
\begin{proof}[Proof of Theorem \ref{main}]
By Proposition \ref{rred}  we may assume $I$ to be $Z$-stable in order to prove our thesis
$$\Hilb\left(H^i_{\mm_{R[Z]}}(R[Z]/I)\right)\leq \Hilb\left(H^i_{\mm_{S[Z]}}(S[Z]/\epsilon_1(I))\right), \hbox{\;\;\; for all } i.$$
\\
\framebox{$i=0$} For this case it is enough to recall that $\Hilb\left(R[Z]/I\right)=\Hilb\left(S[Z]/\epsilon_1(I)\right)$ and Proposition \ref{hyperplane} yields $\Hilb\left(R[Z]/(I+(Z^j))\right)\leq\Hilb\left(S[Z]/(\epsilon_1(I)+(Z^j))\right)$, for all $j$. Thus, for all $j$, $\Hilb(I:Z^j/I)\leq\Hilb\left(\epsilon_1(I):Z^j/\epsilon_1(I)\right)$ which is equivalent to our thesis if $j$ is chosen to be large enough, as we already observed in Remark \ref{ZZm} (b).\\
%NON SERVE COSI: As an immediate consequence we also have  $\Hilb(I\sat)\leq\Hilb\left(\epsilon'(I)\sat\right)$, which we will use later in the proof, for the case $i=1$.\\
\framebox{$i=1$} If $H^1_{\mm_{R[Z]}}(R[Z]/I)=0$ there is nothing to prove. Suppose then that this is not the case.
Now, an application of \eqref{sommaecala} with $d\gg 0$ and for all $j\leq d$  yields
\begin{equation}\label{euna}
\begin{split}
  \Hilb\left(H^1_{\mm_{R[Z]}}(R[Z]/I)\right)_{d-j}&=\sum_{k=0}^j\Hilb\left(H^0_{\mm_{R}}(R/\ov{I\sat})\right)_{d+1-k}\\
&=\underbrace{\sum_{k=0}^j\Hilb\left(\ov{I\sat}\sat\right)_{d+1-k}}_\text{$\Sigma_1(I)$}-\underbrace{\sum_{k=0}^j\Hilb\left(\ov{I\sat}\right)_{d+1-k}}_\text{$\Sigma_2(I)$}.
\end{split}
\end{equation}
We now look at the terms appearing in the first sum. 
%
%NOT NEEDED: Since $\Hilb\left(\ov{I}\right)=\Hilb\left(I+(Z_m)/(Z_m)\right) \geq \Hilb\left(\epsilon'(I)+(Z_m)/(Z_m)\right)=\Hilb\left(\ov{\epsilon_1(I)}\right)$ by Theorem \ref{hyperplane}, we have $\Hilb\left(\epsilon(\ov{I})\right)\geq \Hilb\left(\ov{\epsilon_1(I)}\right)$. 
By Lemma \ref{pregiatolemmaZ2palle},
$\epsilon(\ov{I})\sat$ and $\ov{\epsilon_1(I)}\sat$ are equal, and thus \eqref{filanoesaturano} implies
 \begin{equation}\label{unaemezzo}
\begin{split}
\Sigma_1(I) &=\sum_{k=0}^j\Hilb\left(\ov{I\sat}\sat\right)_{d+1-k}=\sum_{k=0}^j\Hilb\left(\ov{I}\sat\right)_{d+1-k}
\leq \sum_{k=0}^j\Hilb\left(\epsilon(\ov{I})\sat\right)_{d+1-k}\\
&=\sum_{k=0}^j\Hilb\left(\ov{\epsilon_1(I)}\sat\right)_{d+1-k}
=\Sigma_1(\epsilon_1(I)),
\end{split}
\end{equation}
for $\Hilb\left(\ov{I}\sat\right)\leq\Hilb\left(\epsilon(\ov{I})\sat\right)$ descends easily from the fact that $(R,S,\epsilon)$ is cohomology extremal considering cohomological degree $0$. Since  $\Sigma_2(I)\geq \Sigma_2(\epsilon_1(I))$ by Lemma \ref{perH2}, \eqref{unaemezzo} now implies
\begin{equation*}
\begin{split}
  \Hilb\left(H^1_{\mm_{R[Z]}}(R[Z]/I)\right)_{d-j}&=\Sigma_1(I)-\Sigma_2(I)\\
&\leq \Sigma_1(\epsilon_1(I))-\Sigma_2(\epsilon_1(I))) = \Hilb\left(H^1_{\mm_{S[Z]}}(S[Z]/\epsilon_1(I))\right)_{d-j}
\end{split}
\end{equation*}
and this case is  completed.\\
\framebox{$i> 1$} By \eqref{perIm1} and being $\sum_{j<0} t^j$ a series with positive coefficients, we are left to prove the inequality $\Hilb\left(H^i_{\mm_{R}}(R/\ov{I}\sat)\right)\leq
\Hilb\left(H^i_{\mm_{S}}(S/\ov{\epsilon_1(I)}\sat)\right)$ for all $i>0$, or its equivalent $\Hilb\left(H^i_{\mm_{R}}(R/\ov{I})\right)\leq \Hilb\left(H^i_{\mm_{S}}(S/\ov{\epsilon_1(I)})\right)$ for $i>0$. By hypothesis, $\Hilb\left(H^i_{\mm_{R}}(R/\ov{I})\right)\leq \Hilb\left(H^i_{\mm_{S}}(S/\epsilon(\ov{I}))\right)$ for all $i>0$, thus we may conclude if we know, and we do by Lemma \ref{pregiatolemmaZ2palle}, that $\epsilon(\ov{I})$ and $\ov{\epsilon_1(I)}$ have the same saturation.\\
Now the proof of the theorem is complete.
\end{proof}

\section{A Lex-Plus-Power-type inequality for local cohomology}
Let $A=B=K[X_1,\dots,X_n]$, let $\aa=({\bf f})=(f_1,\dots,f_r)$ be the ideal of $A$ generated by a homogeneous regular sequence $f_1,\dots,f_r$ of degrees $d_1\leq \cdots \leq d_r,$ and let $\bb\subseteq B$ be the ideal $(X_1^{d_1},\dots,X_r^{d_r})$. As before, we let $R=A/\aa$ and $S=B/\bb$ and recall that, by Clements-Lindstr\"om Theorem, $S$ has an embedding whose image set consists of the classes in $S$ of all lex-segment ideals of $B$. Henceforth,  such an embedding will be denoted by $\epsilon_{\rm{CL}}$. Thus, we may restate
{\bf EGH} in the following way.
\begin{conjecture}[Eisenbud-Green-Harris]
Let ${\bf f}$ be as above. Then,  $R$ embeds into $(S,\epsilon_{\text{CL}})$.
\end{conjecture}

At the moment there are few cases for which a proof of this conjecture is known, and are essentially contained in  \cite{ClLi}, \cite{CaMa}, \cite{Ch} and \cite{Ab}. Yet, Evans wondered if the following far-reaching result on graded Betti numbers holds. 

\begin{conjecture}[Evans' {\bf LPP}] Assume that ${\bf f}$ satisfies {\bf EGH}. Then, for all homogeneous ideal $I$ of $R$ and all $i$
$$\beta^A_{ij}(R/I):=\Hilb\left(\Tor_i^A(R/I,K)\right)_j\leq \Hilb\left(\Tor_i^B(S/\epsilon_{\rm{CL}}(I),K)\right)_j=:\beta^B_{ij}(S/\epsilon_{\rm{CL}}(I)).$$
\end{conjecture}

\begin{remark} Consistently with our definition of cohomology extremal embeddings, we will call an embedding 
$(R,S,\epsilon)$ satisfying the inequality predicted by the {\bf LPP} conjecture above, \emph{Betti extremal}. With this terminology, Theorem 3.1 of \cite{CaKu2} together with the subsequent discussion imply that, when $(R,S,\epsilon)$ is Betti extremal, $(R[Z],S[Z],\epsilon_1)$ is Betti extremal as well. 
\end{remark} 

\noindent
The only case in which {\bf LPP} is known so far is when ${\bf f}$ is monomial, see \cite{MePeSt} for a proof when $d_1=\cdots=d_r=2$ and \cite{MeMu} for a proof without restrictions on the degrees. The purpose of this section is to prove in Theorem \ref{mainLPP} an analogous of the {\bf LPP} conjecture  when we consider local cohomology modules instead of $\Tor$ modules.  Theorem \ref{mainLPP} holds not only when ${\bf f}$ is monomial but also in all the cases for which {\bf EGH} is known. More precisely, our assumption is to require that at least an Artinian reduction of $A/({\bf f})$ satisfies {\bf EGH}.

Let $l_1,\dots,l_{n-r}$ be a sequence of linear forms such that $\mathbf {f},l_1,\dots,l_{n-r}$ form an $A$-regular sequence, which always exists provided that $K$ is infinite.  After applying a coordinates change we may assume these linear forms to be  $X_n,\dots,X_{r+1}.$ Let ${\ov{\bf f}}\in \ov{A} =K[X_1,\dots,X_r]$ be  the image of ${\bf f}$ modulo $X_n,\dots, X_{r+1}$. We also let $\ov{B}=\ov{A}$ and $\ov{\bb}$ be the image of $\bb$ in $\ov{B}$. Finally, $\ov{\epsilon}_{\rm CL}$ will denote the Clements-Lindstr\"om embedding of $\ov{B}/\ov{\bb}$.

\begin{theorem}\label{mainLPP} Assume that ${\ov{\bf f}}$ satisfies {\bf EGH}. Then, for all homogeneous ideal $I$ of $R$. 
$$
\Hilb\left(H^i_{\mm_{R}}(R/I)\right) \leq \Hilb\left(H^i_{\mm_{S}}(S/\epsilon_{\rm{CL}}(I))\right). 
$$ 
In other words, $(R,S,\epsilon_{\rm{CL}}(I))$ is cohomology extremal.
\end{theorem}

\begin{proof} Let $I$ be a homogeneous ideal of $R$ and let $J$ denote its pre-image in $A$. Clearly, $({\bf f})\subseteq J$ and, if we let $\omega$ be the weight vector with entries $\omega_i=1$ for all $i\leq r$ and $0$ otherwise, then ${\bf f}\in P:=\ini_\omega(J)$. By hypothesis $\ov{A}/({\ov{\bf f}})$ embeds into $(\ov{B}/\ov{\bb}, \ov{\epsilon}_{\rm{CL}})$, and moreover being both rings Artinian $\left(\ov{A}/({\ov{\bf f}}),\ov{B}/\ov{\bb}, \ov{\epsilon}_{\rm{CL}}\right)$ is trivially cohomology extremal. Theorem \ref{main} now yields that $(A/({\ov{\bf f}})A,S, \epsilon)$ is also cohomology extremal and $\epsilon$, which is obtained by extending $\ov{\epsilon}_{\rm{CL}}$, is precisely the Clements-Lindstr\"om embedding $\epsilon_{\rm CL}$ of $S$, cf. Example \ref{thhree} part (ii). By Base Independence of local cohomology and by \cite{Sb} Theorem 2.4,
$\Hilb\left(H^i_{\mm_{R}}(R/I)\right) = \Hilb\left(H^i_{\mm_{A}}(A/J)\right)\leq \Hilb\left(H^i_{\mm_{A/(\ov{\bf f})A}}(A/P)\right)$. 
%qui c'e' un altro uso di base independence.
Finally, since $I$ has the same Hilbert function as the image  of $P$ in $A/(\ov{\bf f})A$ and  $\left(A/({\ov{\bf f}})A,S, \epsilon_{\rm{CL}}\right)$ is cohomology extremal, we have $\Hilb\left(H^i_{\mm_{A/(\ov{\bf f})A}}(A/P)\right)$ $\leq$ $\Hilb\left(H^i_{\mm_{S}}(S/\epsilon_{\rm{CL}}(I)))\right)$ as desired.  
\end{proof}

\begin{remark}
Clearly, if {\bf EGH} were true in general then the assumption on ${\ov{\bf f}}$ in Theorem \ref{mainLPP} would be trivially satisfied and  $(R,S,\epsilon_{\rm{CL}})$ would  be cohomology extremal. It is proven in \cite{CaMa} that if ${\ov{\bf f}}$ satisfies {\bf EGH} then  ${\bf f}$ does, whereas at this point we do not know about the converse.  %We also would like to point out that in all the currently known cases of sequences ${\bf f}$ which satisfy {\bf EGH}, there  exists a regular sequence of linear forms such that the induced sequence ${\ov{\bf f}}$ which defines the Artinian reduction satisfies {\bf EGH} as well. For instance, this is the case when ${\bf f}$ is monomial as in \cite{ClLi}.  
A simple flat deformation argument together with the results of \cite{MeMu} and \cite{MePeSt} shows that {\bf LPP} holds true when ${\bf f}$ is a Gr\"obner basis with respect to some given term order $\tau$. We would like to point out that, under this assumption, the conclusion of Theorem \ref{mainLPP} holds as well, and we prove our claim in the following lines.  
Let ${\bf g}=\ini_\tau(f_1),\ldots,\ini_\tau(f_r)$. As in the proof of Theorem \ref{mainLPP}, by \cite{Sb} Theorem 2.4 it is sufficient to bound above $\Hilb\left(H^i_{\mm_{A/({\bf g})}}\left(A/\ini_\tau(({\bf f})+I)\right)\right)$. By hypothesis ${\bf g}$ form a monomial regular sequence in $A$, therefore it is easy to see that there are linear forms $l_1,\ldots,l_{n-r}$ such that $A/({\bf g},l_1,\ldots,l_{n-r})$ is Artinian and isomorphic to 
%le forme lineari sono X_{prima lettera del supporto di \ini(f_1)}-X_{altre nello stesso supporto} e via dicendo, in caso avessi altre variabili le ammazzo.
$K[X_1,\ldots,X_r]/(X_1^{d_1},\ldots,X_r^{d_r})$, which has the Clements-Lindstr\"om embedding. Therefore, the sequence $\ov{\bf g}$ in $A/(l_1,\ldots,l_{n-r})$ satisfies {\bf EGH}, and thus, verifies the hypothesis of Theorem \ref{mainLPP}, which now yields what we wanted. 
\end{remark}

\section{Extremal Betti numbers and {\bf LPP}}
In this section we show how to derive directly from Theorem \ref{mainLPP} a special case of {\bf LPP} for Betti numbers. Precisely, we prove that the inequality predicted by {\bf LPP} holds for those  Betti numbers which in the literature,  following  \cite{BaChPo}, are called {\em extremal}.
Furthermore, we will show an inclusion between the regions of the Betti tables of $R/I$ and $S/\epsilon_{\rm{CL}}$ where non-zero values may appear, and which are outlined by the positions of the corresponding extremal Betti numbers. 

As in the previous section, let ${\bf f}\in A=K[X_1,\dots,X_n]$ be a homogeneous regular sequence of  degrees $d_1\leq \dots \leq d_r$. By extending the field we may assume that $\vert K \vert = \infty$ and up to a change of coordinates, that ${\bf f}, X_n,\dots, X_{n-r+1}$ form a regular sequence as well. 
Let $M$ be a finitely generated graded $A$-module; the following definition was introduced in \cite{BaChPo} when $M=A/I$. A non-zero $\beta^A_{ij}(M)=\dim_K \Tor^A_i(M,K)_j$ such that $\beta_{rs}(M)=0$ whenever $r\geq i, s\geq j+1$ and  $s-r \geq j-i$ is called {\em extremal Betti number} of $M$. A pair of indexes $(i,j-i)$ such that $\beta_{ij}(M)$ is extremal is called a {\em corner} of $M$.  The reason for this terminology is that extremal Betti numbers correspond to certain corners in the output of Macaulay2 \cite{M2} command for computing Betti diagrams.
Let $I$ be a homogeneous ideal of $A$ and denote by $\Gin(I)$ the generic initial ideal of $I$ with respect to the reverse lexicographic order, cf. \cite{Gr2} and \cite{Ei} for more details on generic initial ideals. The interest in  extremal Betti numbers comes from the fact, proved in \cite{BaChPo} and  \cite{Tr}, that $A/I$ and $A/\Gin(I)$ have the same extremal Betti numbers, and therefore same corners. Since projective dimension and Castelnuovo-Mumford regularity can be computed from corners, this result is a strengthening of the well-known Bayer-Stillman Criterion \cite{BaSt}. In the proof, one can use the fact that the extremal Betti numbers of $M$ can be computed directly from the Hilbert functions of certain local cohomology modules, and in particular from the considerations in \cite{BaChPo} or \cite{Tr} one can deduce that, for any finitely generated graded $A$-module $M$
\begin{equation}\label{COR}
\beta_{ij}^A(M)=\Hilb\left(H^{n-i}_{\mm_A}(M)\right)_{j-n}, 
\end{equation}
when $(i,j-i)$ is a corner of $M$.

For the proof of the next theorem, we need to recall the definition of  partial Castelnuovo-Mumford regularities and its characterization: Given a finitely generated $A$-module $M$, for any integer $0\leq h \leq \dim M$, we have  
 \[\reg_h(M):= \sup \{j-i :  \beta_{ij}^A(M)\not = 0, i\geq n-h\}= \sup \{j+i : (H^{i}_{\mm_A}(M))_{j}\not =0 , i\leq h \},\]
see \cite{Tr} Theorem 3.1 (i).
 If we  set $\reg_{-1}(M)= - \infty,$ clearly we have $\reg_{-1}(M)\leq \reg_0(M) \leq \reg_1(M)\leq \cdots \leq \reg_n(M)=\reg(M)$. Moreover, the corners of the Betti table of $M$ can be determined by looking at the strict inequalities in the previous sequence: $(i,j-i)$ is a corner of $M$ if and only if
\begin{equation}\label{corners}
\reg_{n-i-1}(M) < \reg_{n-i}(M)
\quad \text{ and }  \quad 
j-i=\reg_{n-i}(M);
\end{equation}  
in particular 
\begin{equation}\label{reg-cor}
\reg_h(M)= \sup \{j-i : (i,j-i) \text{ is a corner of } M \text{ and }  i\geq n-h\}. 
\end{equation}

\begin{theorem}[{\bf LPP} for extremal Betti numbers]\label{extremalLPP}
Let $\ov{\bf f}$ satisfy {\bf EGH}. Then, for all homogeneous ideals $I$ of $R$, 
$$\beta^A_{ij}(R/I)\leq \beta^A_{ij}(S/\epsilon_{\rm{CL}}(I)),$$
when $(i,j-i)$ is a corner of $S/\epsilon_{\rm{CL}}(I).$
\end{theorem}
\begin{proof} If $n=1$ there is nothing to prove. Thus, let $n\geq 2$ and observe that, if $(i,j-i)$ is also a corner of  $R/I$ then the conclusion is straightforward by the use of Theorem \ref{mainLPP} and \eqref{COR}. Otherwise, since  Theorem \ref{mainLPP} yields $\reg_h(R/I)\leq \reg_h(S/\epsilon_{\rm{CL}}(I))$ for all $h$,  then $\reg_{n-i}(R/I)<\reg_{n-i}(S/\epsilon_{\rm{CL}}(I))=j-i.$ Hence $\beta^A_{ij}(R/I)=0.$
\end{proof}

Furthermore, under the same assumption of the above theorem, we have the following result, see also Figure \ref{bettitab}.

\begin{figure}[here] 
\begin{center}
\includegraphics[scale=0.15]{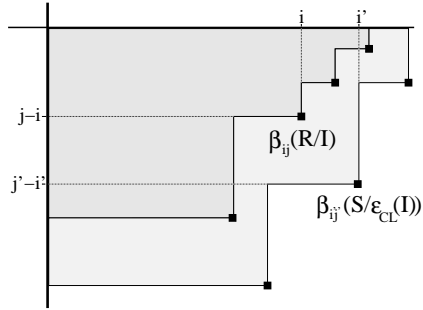}
\end{center}
\caption{The Betti diagrams of $R/I$ and $S/\epsilon_{\rm{CL}}(I))$, and their extremal Betti numbers.}
\label{bettitab}
\end{figure}
 
\begin{theorem}[Inclusion of Betti regions]\label{inc-region} Let $(i,j-i)$  be a corner of $R/I$. Then there exists 
a corner  $(i',j'-i')$ of  $S/\epsilon_{\rm{CL}}(I))$ such that  $i\leq i'$ and $j-i \leq j'-i'.$
\end{theorem}
\begin{proof} Theorem \ref{mainLPP} implies that $\reg_h(R/I)\leq \reg_h(S/\epsilon_{\rm{CL}}(I))$ for all $h$. Since $(i,j-i)$ is a corner of $R/I$, by \eqref{corners} $\reg_{n-i}(R/I)=j-i$ and, therefore,  $\reg_{n-i}(S/\epsilon_{\rm{CL}}(I))\geq j-i.$   By  \eqref{reg-cor} we know that $\reg_{n-i}(S/\epsilon_{\rm{CL}}(I))= j'-i'$ for some corner $(i',j'-i')$ of  $S/\epsilon_{\rm{CL}}(I)$ satisfying $i'\geq n-(n-i)=i$, as we desired.
\end{proof}

%--------------------\\fino a qui\\------------------------\\
%

\end{document}